\documentclass[reqno]{amsart}

\usepackage{verbatim}
\usepackage{amsmath}
\usepackage{color}
\usepackage{url}
\usepackage{enumitem}
\usepackage{lineno,xcolor}
\usepackage{graphicx} 
\usepackage{hyperref}
\hypersetup{colorlinks}

\usepackage{amsmath,amssymb}
\usepackage{graphicx}

\title[Necklaces and multisets]{A bijection between necklaces and  multisets  with divisible subset  sum}

\author{Swee Hong Chan}
 \address{Department of Mathematics, Cornell University, Ithaca, NY 14853.}
\email{\url{sweehong@math.cornell.edu}}
\urladdr{https://www.math.cornell.edu/~sc2637/} % Delete if not wanted.

\newtheorem{theorem}{Theorem}[section]
\newtheorem{lemma}[theorem]{Lemma}

\theoremstyle{definition}

\theoremstyle{definition}

\theoremstyle{definition}

\newtheorem{examplex}[theorem]{Example}
\newenvironment{example}
  {\pushQED{\qed}\examplex}
  {\popQED\endexamplex}

   \newtheorem{definitionx}[theorem]{Definition}
\newenvironment{definition}
  {\pushQED{\qed}\definitionx}
  {\popQED\enddefinitionx}

\newcommand{\lcm}{\text{lcm}}% least common multiple
\newcommand{\Nc}{\mathcal N} % the set of cyclic necklaces
\newcommand{\Ec}{\mathcal E} % the set of all functions from {0,1,\ldots,q-1} 
\newcommand{\Fc}{\mathcal F} % the set of all functions with zero sum
\newcommand{\Qc}{\mathcal Q} % the  polynomial ring over F_2 modulo x^n-1
\newcommand{\F}{\mathbb{F}} %field
\newcommand{\Z}{\mathbb{Z}} %integers 
\begin{document}

\begin{abstract}
Consider these two distinct combinatorial objects:
(1) the  necklaces  of length $n$ with at most $q$ colors,
and  (2) the multisets of integers modulo $n$
with subset sum divisible by $n$
and with 
the  multiplicity of each element being strictly less than $q$.
%the functions that maps integer modulo $n$ to nonnegative integers strictly less than $q$ with the property that
% their linearly weighted sum  is divisible by $n$.
We show that these two objects have the same cardinality if $q$ and $n$ are mutually coprime.
Additionally, when  $q$ is a prime power, we construct a bijection between these two objects by viewing  necklaces as  cyclic polynomials over the finite field of size $q$.
Specializing to $q=2$ 
answers a  bijective problem posed by Richard  Stanley (Enumerative Combinatorics Vol. 1 Chapter 1, Problem 105(b)).
%\keywords{Necklaces, subset sums, cyclic polynomials}
\end{abstract}



\maketitle

%\maketitle
%\setpagewiselinenumbers % give line numbers for the ease of review
%\modulolinenumbers[1]
%\linenumbers
%

\section{Introduction}\label{section: intro}
Let $q$ be and $n$ be two coprime positive integers.
The main characters of this paper are the following two combinatorial objects:
\begin{itemize}
\item The set  $\Nc$  of  necklaces (i.e., equivalent up to cyclic rotations) of length $n$ for which the color of each bead is drawn from a color set of size $q$.

\item The set $\Fc$ of functions $f:\Z_n \to \{0,1,\ldots,q-1\}$ for which their (linearly) weighted sum is divisible by $n$, i.e., 
\[\Fc:=  \left \{ f \ \bigg | \ \sum_{z \in \Z_n} z\, f(z) = 0 \quad \left( \text{mod } n \right) \right \}, \]
where  $\Z_n$ denotes the ring of integers modulo $n$.
\end{itemize}
Equivalently, $\Fc$ is the set of multisets of $\Z_n$ with   subset sum  divisible by $n$ and with  the   multiplicity of each element being at most $q-1$. 
%When $q=2$, the set $\Fc$ corresponds to  the set of subsets of $\Z_n$ whose subset sum is divisible by $n$.
The set $\Fc$ for the case $q=2$ has been studied in  different areas of mathematics,
such as coding theory~\cite{SY72}, number theory~\cite{OS78}, and  toric arrangements~\cite{ACH15,AC17}.

%When $q=2$,
% $\Fc$ can be viewed 
%and  $\Nc$ is the set of necklaces of length $n$ with two colors.
It was known  that $\Nc$ and $\Fc$ have the same cardinality when $q=2$ (see  \cite[Problem 105(b) Chapter 1]{Stan12}).
We extend this result to all values of $q$.

\begin{theorem}\label{theorem: combinatorial}
Let $q$ and $n$ be two coprime positive integers.
Then 
\[ |\Nc|=|\Fc|=\sum_{I \subseteq \{1,\ldots,m\}} \frac{\gcd(n, \gcd(s_i)_{i \in I})}{n}   \, \prod_{i \in I} (q^{\ell_i}-1),  \]
where $m$, $s_i$, and $\ell_i$ are as in Definition~\ref{definition: S_i}.
\end{theorem}

%Theorem~\ref{theorem: combinatorial} is proved in \S\ref{subsection: proof of main theorem 1}.
We remark that  Theorem~\ref{theorem: combinatorial} gives a new expression 
for the cardinality of  $\Nc$ and $\Fc$.
This expression      is different from the formulas in \cite[Theorem~11]{KP93} and \cite[Section~4.2]{Kusz14}, which involve the M\"obius function and the Euler's totient function.
We also remark that  
 the condition that $n$ and $q$ are coprime  is necessary, as there are examples for which 
$|\Nc|$ is not equal to  $|\Fc|$ when $\gcd(n,q)>1$.
One such example is when $n=q=2$, which gives us  $|\Nc|=3$ and $|\Fc|=2$.

The proof of Stanley  for the case $q=2$ is not bijective in nature, 
and neither is our proof of Theorem~\ref{theorem: combinatorial}.
In \cite[Problem 105(b) Chapter 1]{Stan12}, Stanley asked for a bijective proof of Theorem~\ref{theorem: combinatorial} for the case $q=2$.
 We answer this question here by constructing  a bijection   between the two sets when $q$ is a prime power.

%We will  create partitions 
% $\{ \Nc_I \}_{I \subseteq \{1,\ldots,m\}}$ 
%and $\{ \Fc_I \}_{I \subseteq \{1,\ldots,m\}}$ of  $\Nc$ and $\Fc$, respectively (Definition~\ref{definition: Sc^I and Nc^I}).
%We then construct  maps $\psi_I: \Nc_I \to \Fc_I$ (Definition~\ref{definition: bijection }), one for each $I \subseteq \{1,\ldots,m\}$.

Our bijection  starts by viewing necklaces with $q$ colors as cyclic polynomials over the finite field $\mathbb{F}_q$.
Each necklace   can then be associated to a coset of a  finite abelian group
by taking the remainder of the division of the cylic polynomial by irreducible factors of $X^n-1$.
On the other hand, a function in $\Fc$ can be associated to an element of the same finite abelian group
by evaluating the function on the cyclotomic cosets of $\Z_n$.
It will follow from the construction  that, for any given necklace,  the corresponding coset  contains exactly one group element that is associated to a function in $\Fc$.
 We take this unique function as the image of the necklace under our bijection. 
The full definition of this bijection is given in \S\ref{subsection: proof of main theorem 2}.

\begin{theorem}\label{theorem: bijection}
Let $q$ be a prime power, and   let $n$ be a positive  integer that is coprime to $q$.
Then the map $\widehat{\psi}: \Nc \to \Fc$  in Definition~\ref{definition: bijection}  is a bijection.
\end{theorem}
%This gives us the desired  bijection from $\Nc$ to $\Fc$.
%The map $\psi$  viewing necklaces with $q$ colors as cyclic polynomials over the finite field $\mathbb{F}_q$.
%Theorem~\ref{theorem: bijection} is proved in \S\ref{subsection: proof of main theorem 2}.
See Example~\ref{example: 5} for
an example of the bijection $\widehat{\psi}$  when $q=2$ and $n=3$.
A bijection for general values of $q$  
remains an open problem.
%and therefore answers 
% the problem in \cite[Problem 105(b) Chapter 1]{Stan12}.

This paper is structured as follows.
In \S\ref{s. notation}, we review algebraic tools that will be used in the proofs of the main theorems.
In \S\ref{subsection: proof of main theorem 1},
we present a proof of Theorem~\ref{theorem: combinatorial} in \S\ref{subsection: proof of main theorem 1}.
 In \S\ref{subsection: proof of main theorem 2}, we present  a proof of Theorem~\ref{theorem: bijection}.
 In \S\ref{section: unsolved questions}, we present two open  bijective problems that extend
 Theorem~\ref{theorem: bijection}.

%In \cite[Problem 105(b) Chapter 1]{Stan12}, Stanley showed that $\Nc$ and $\Fc$ have the same cardinality.
%The proof in \cite{Stan12}  is not bijective in nature, and
%a bijective proof of this fact is left as an open problem.
%We answer this question in this paper by constructing  a bijection   between the two sets. 
%The definition of the bijection $\psi$ is given in Section~.

\section{Preliminaries} \label{s. notation}
In this section, we review  algebraic tools that will be used in the proof of Theorem~\ref{theorem: combinatorial} and Theorem~\ref{theorem: bijection}.
%Throughout this paper we will use the following sets to partition $\Z_n$.
%We refer to Example~\ref{example: 1} for concrete examples of the algebraic objects in this section.

Throughout this paper, $q$ and $n$ are two positive integers such that $\gcd(n,q)=1$.
\begin{definition}\label{definition: S_i}
%Let $n$ be an odd positive integer.
Consider the equivalence relation on $\Z_n$
that takes all multiplications by $q$ as equivalent. 
 Fix integers $s_1,s_2,\ldots, s_m$ as the representatives of the equivalence classes of this relation.
The \emph{cyclotomic cosets} $S_1,\ldots, S_m$ of $\Z_n$ are 
\[  S_i:=\{s_i,q\, s_i,q^2 s_i,\ldots, q^{\ell_i-1}s_i\} \qquad (i\in \{1,\ldots,m\}), \]
where    $\ell_i$ is the smallest positive integer such that 
$q^{\ell_i}s_i = s_i \text{ (mod n)}$. 
\end{definition}

%We will  
%use  integers in $\bigcup_{i=1}^m S_i$ as the representative of elements  in $\Z_n$ instead of the more common choice of $\Z_n=\{0,1,\ldots,n-1\}$ for a reason that will be apparent soon.

When $q$ is a prime power, we view the set of necklaces $\Nc$  from  the following algebraic perspective.
Let $\Qc$ be 
the quotient
\[  \Qc:= \frac{\F_q[X]}{(X^n-1)},\]
of the
polynomial ring over the finite field $\F_q$ of order $q$ in a single variable $X$
by the ideal generated by $X^n - 1$.
 Each element of $\Qc$ corresponds to an $n$-character string over an alphabet of size $q$ by taking its coefficient vector.
The set $\Nc$  can then be viewed as
\[ \Nc:=\left\{  \{ \alpha ,X\alpha,\ldots, X^{n-1} \alpha  \}  \  \bigg | \ \alpha \in \Qc \right \},  \]
 the set of  equivalence classes of the relation in  $\Qc$ that takes all multiplications by $X$ as equivalent.

%Let $q$ be a prime power.
Fix  a  primitive $n$-th root of unity  $\omega$ in the algebraic closure of  $\F_q$.
 Such $\omega$ exists because $q$ is coprime to $n$.
%Write $\{1,\ldots,m\}:=\{1,\ldots,m\}$.

\begin{definition}\label{definition: P_i}
Let $q$ be a prime power.
Let $P_1,\ldots, P_m$ be the irreducible factors of $X^n-1$ over the field $\F_q$.
That is, for any $i\in \{1,\ldots,m\}$, 
\[ P_i:=\prod_{k \in S_i} (X- \omega^{k}). \]
We denote by $G_i$ the set
\[ G_i:= (\Qc/P_i \Qc)^\times, \]
 of nonzero elements of the quotient ring $\Qc/P_i \Qc$. 
\end{definition}
%Equivalently, $P_i$ is the element of $\Qc$ that 
%Note that $P_i$ is contained in $\Qc$ as the polynomial in Definition~\ref{definition: P_i} is the irreducible monic polynomial in $\F_q[x]$ that has $\omega^{s_i}$ as a root.

\begin{definition}\label{definition: alpha_i}
Let $q$ be a prime power.
For any $\alpha \in \Qc$, we denote by 
$\alpha_i:= \alpha \mod P_i$ the image of $\alpha$ in $\Qc/ P_i \Qc$ under the quotient map.
In particular, $X_i$ is the image of $X$  in $\Qc/ P_i \Qc$. 
\end{definition}

We now present examples of the   objects discussed above for the case that $q=2$ and $n=3$.
This case will be our running example throughout this paper.
\begin{example}\label{example: 0}
Let $q=2$ and $n=3$. 
We make the following choices of cyclotomic cosets from Definition~\ref{definition: S_i}:
\begin{align*}
& s_1=0, \quad S_1=\{ 0 \}; \quad \text{ and } \quad  s_2=1, \quad S_2=\{1,2\}.
\end{align*}
%Using this choice of $t_i$'s, the set $\Z_5$ is then given by:
%\[\Z_5=\{0,1,2,4,8\}.   \]

We represent a function $f:\Z_3 \to \{0,1\}$ 
as the set   $\{ z \in \Z_3 \mid f(z)=1 \}$.
In this notation, the  sets  $\Nc$ and $\Fc$ are  given by
\begin{align*}
\Nc=&\{ \{0\}, \{1,X,X^2 \}, \{1+X, X+X^2,1+X^2  \}, \{1+X+X^2   \}\},\\
\Fc=&\{ \varnothing,  \{0\},  \{1,2\}, \{0,1,2\} \}.
\end{align*}
%We make the following choice for $t_i$ and $T_i$:
%\begin{align*}
%& t_1=0, \quad T_1=\{ 0 \}; \\
%&t_2=1, \quad T_2=\{1,2,4,8\}= \{1,2,3,4\} \quad \text{mod } 5.
%\end{align*}
The polynomials  $P_i\in \Qc$ from Definition~\ref{definition: P_i} are  given by
\[P_1=1+X; \qquad P_2=1+X+X^2. \qedhere \]
%The sets $\Fc_I$ and $\Nc_I$ are then given by
%\begin{alignat*}{2}
%&  \Nc_\varnothing=\{\{0\}\}, \qquad &&\Fc_\varnothing=\{\{0,1,2 \} \};\\
%&\Nc_{\{1\}}=\{ \{ 1+X+X^2  \}\}, \qquad  &&\Fc_{\{1\}}=\{ \{1,2  \} \}; \\
%&\Nc_{\{2\}}=\{ \{1+X,X+X^2,1+X^2  \}  \}, \qquad &&\Fc_{\{2\}}=\{ \{0 \} \};\\
%&\Nc_{\{1,2\}}=\{   \{1,X,X^2  \}  \}, \qquad &&\Fc_{\{1,2\}}=\{ \varnothing \}. \tag*{\qedhere} 
%\end{alignat*}
\end{example}

%Let $F_i^\times$ denote the set of invertible elements in $F_i$.
 
We refer to \cite{Wan03} for  the proofs of  the following properties of $\Qc/P_i \Qc$ and  $G_i$.

\begin{lemma}[{\cite[Section~9]{Wan03}}]\label{lemma: finite field}
Let $q$ be a prime power, and let $n$ be a positive integer coprime to $q$. 
For any $i \in \{1,\ldots,m\}$,
\begin{enumerate}[{
label=\textnormal{({\roman*})}, labelindent=0pt,
ref={\roman*}}]
\item \label{item: finite field 1}  $\Qc/P_i \Qc$ is a finite field of order $q^{\ell_i}$.
\item \label{item: finite field 2}  $G_i$  is a cyclic group of order $q^{\ell_i}-1$ under  multiplication.
\item \label{item: finite field 3} $X_i$ is an element of $G_i$ with multiplicative    order $\frac{n}{\gcd(n,s_i)}$.
\qed
\end{enumerate}
\end{lemma}

%We use the case $n=5$ as the running example of this paper.

We will use the following  versions of the Chinese remainder theorem in the proof of Theorem~\ref{theorem: combinatorial} and Theorem~\ref{theorem: bijection}.

\begin{theorem}[{Chinese remainder theorem~\cite[Theorem~2.25]{Hun80}}]
\label{theorem: Chinese remainder theorem} 
 \ 
\begin{enumerate}
[label=\textnormal{({\roman*})}, labelindent=0pt,
ref={\roman*}]
\item \label{item: CRT1}
Let $n$ be a positive integer with prime factorization $n=p_1^{a_1}\ldots p_\ell^{a_\ell}$.
Then the following map is an isomorphism:
\begin{align*}
\Z/ n\Z & \to  \Z/p_1^{a_1}\Z \times \ldots \times \Z/p_\ell^{a_\ell}\Z\\
x \textnormal{ mod } n & \mapsto ( x \textnormal{ mod } p_1^{a_1},\ldots, x \textnormal{ mod } p_\ell^{a_\ell}).
\end{align*}

\item \label{item: CRT2}
Let $q$ be a prime power and let $n$ be a positive integer coprime to $q$.
Then the following map is an isomorphism:
\begin{align*}
\Qc & \to  \Qc/P_1\Qc \times \ldots \times \Qc/P_m\Qc\\
\alpha  & \mapsto ( \alpha \textnormal{ mod } P_1,\ldots, \alpha \textnormal{ mod } P_m).  \tag*{\qed}
\end{align*}
\end{enumerate}
\end{theorem}

%\begin{theorem}[{Chinese remainder theorem~\cite[Theorem~2.25]{Hun80}}]
%\label{theorem: Chinese remainder theorem}
%
%Let $q$ be a prime power and let $n$ be a positive integer coprime to $q$.
%Then the following map is an isomorphism:
%\begin{align*}
%\Qc & \to  \Qc/P_1\Qc \times \ldots \times \Qc/P_m\Qc\\
%\alpha  & \mapsto ( \alpha \textnormal{ mod } P_1,\ldots, \alpha \textnormal{ mod } P_m).  \tag*{\qed}
%\end{align*}
%\end{theorem}

%We will use the following consequence  of the Chinese remainder theorem.

The following lemma is a consequence of Theorem~\ref{theorem: Chinese remainder theorem}\eqref{item: CRT1}.
%
%We will use the following lemma in constructing  the bijection in Theorem~\ref{theorem: bijection}.

\begin{lemma}\label{lemma: Bezout variant}
Let $n$ and  $d_1,\ldots, d_k$ be positive integers.
Then there exists a  group automorphism $\phi: \prod_{i=1}^k\Z_{\frac{n}{\gcd(n,d_i)}} \to \prod_{i=1}^k \Z_{\frac{n}{\gcd(n,d_i)}}$ such that 
\begin{equation*}
d_1 h_1+\ldots + d_k h_k= \gcd(n,d_1,\ldots, d_k) \quad \textnormal{(mod $n$)},
\end{equation*} 
where $h_i$ is the $i$-th coordinate of $\phi(1,\ldots,1)$.
\end{lemma}

\begin{proof}
%Let $n=p_1^{a_1}\ldots p_{\ell}^{a_\ell}$ be the prime factorization of $n$.
%By Theorem~\ref{theorem: Chinese remainder theorem}\eqref{item: CRT1},
%it suffices to find an isomorphism $\phi$ that satisfies:
%\[ s_1\phi(e_1)+\ldots + s_k\phi(e_k)= \gcd(s_1,\ldots, s_k) \quad \textnormal{(mod $p_j^{a_j}$)},\]
%for $j \in \{1,\ldots,\ell\}$.
%
%For any $i\in \{1,\ldots,k\}$, let $\frac{n}{\gcd(n,s_i)}:=p_1^{a_{1,i}}\ldots p_{\ell}^{a_{\ell,i}}$ be the prime factorization of $\frac{n}{\gcd(n,s_i)}$.
%Let $\tau_i:\Z_{\frac{n}{\gcd(n,s_i)}}\to \prod_{j=1}^\ell \Z_{p_j^{a_{j,i}}}$
%be the group isomorphism given by Theorem~\ref{theorem: Chinese remainder theorem}\eqref{item: CRT1}.
%Note that
%\[ s_i\psi(e_i)=\sum_{j=1} s_i \]
By the Chinese remainder theorem~(Theorem~\ref{theorem: Chinese remainder theorem}\eqref{item: CRT1}),
 the group and the sum in the lemma can be decomposed into their corresponding prime parts.
Therefore,
it suffices to prove the lemma for when $n$ is a prime power $p^a$.

%Note that the statement of the lemma is vacuously true when $k=0$.
%Assume that $k\geq 1$ throughout the rest of this proof.
For any $i \in \{1,\ldots,k\}$, let $a_i$ be the integer such that $p^{a_i}= \gcd(n,d_i)$, and let
$t_i$ be an integer coprime to $n$ such that
$t_id_i = \gcd(n,d_i) $\ (mod $n$).
% $t_i:=\frac{d_i}{\gcd(n,d_i)}$.
%Note that $t_i$ is coprime to $n$.
%Let $t_i'$ be an integer coprime to $n$
%such that $t_it_i' = 1 $(mod $n$).
Note that $a_i\leq a$ by definition.
By reindexing if necessary, 
we can  without loss of generality assume that
$a_1\leq\ldots\leq a_k$.

Let $e_i$ be the group element $(\underbrace{0,\ldots,0}_{i-1},1,0,\ldots,0)$.
We define $\phi(e_i)$ to be 
%Let $j$ be the largest number in $\{1,\ldots,k\}$ such that $a_j=a_1$.
%We define  $\phi: \prod_{i=1}^k\Z_{\frac{n}{\gcd(n,d_i)}} \to \prod_{i=1}^k \Z_{\frac{n}{\gcd(n,d_i)}}$ by 
\begin{align*}
\phi(e_i):=
\begin{cases}
\displaystyle t_1e_1 -\sum_{j=2}^k e_j &\text{ if } i=1;\\
 e_i & \text{if  } i\in \{2,\ldots,k\}.
 \end{cases}
\end{align*}
%where $e_i$ is the group element $(\underbrace{0,\ldots,0}_{i-1},1,0,\ldots,0)$.
We claim that  $\phi$ can be extended to a group automorphism of $\prod_{i=1}^k\Z_{\frac{n}{\gcd(n,d_i)}}$.

Since $a_1$ is chosen to be the minimum value of $a_i$'s, we have
\[\frac{n}{\gcd(n,d_i)} \, \phi(e_i)=\begin{cases} \displaystyle  t_1 \left( p^{a-a_1} e_1 \right) -\sum_{j=2}^k  p^{a_j-a_1} \left(p^{a-a_j} e_j \right)= 0 & \text{if $i=1$};\\
p^{a-a_i}e_i= 0 &\text{if } i \in \{2,\ldots,k\},
\end{cases}
\]
%Note that $\phi(\frac{n}{\gcd(n,d_i)}e_i)=0$ for all $i\in \{1,\ldots,k\}$ since $a_1,\ldots, a_n$ is a non-decreasing sequence.
 and so $\phi$ extends to a group homomorphism.  
 
Tha map $\phi$ is an automorphism since the  
 corresponding matrix is triangular and all the diagonal entries are coprime to $n$.
Finally, we have
\begin{align*}
\phi(1,\ldots,1)= \sum_{j=1}^k \phi(e_i) =t_1e_1- \sum_{j=2}^k e_j +  \sum_{j=2}^k e_j=t_1e_1=(t_1,0,\ldots,0),
\end{align*}
which implies that
\begin{align*}
d_1 h_1+\ldots + d_k h_k=d_1t_1=\gcd(n,d_1) \quad \text{(mod $n$)}
=& \gcd(n,d_1,\ldots, d_k) \quad \text{(mod $n$)},
\end{align*}
where the last equality is a consequence of  $a_1$ being the minimum value of $a_i$'s.
This proves the claim.
\end{proof}

We will use the following version of Dirichlet's prime number theorem  in the proof of Theorem~\ref{theorem: combinatorial}.

\begin{theorem}[{\cite[Dirichlet's prime number theorem]{JJ98}}]\label{theorem: Bezout Dirichlet}
%\begin{enumerate}[{
%label=\textnormal{({\roman*})}, labelindent=0pt,
%ref={\roman*}}]
%\item \label{item: Bezout} \textnormal{(Bez\'out's identity)\textbf{.}}
%Let $a_1,\ldots, a_k$ be  integers.
%Then there exist integers $c_1,\ldots, c_k$ such that
%\[ a_1c_1+\ldots+ a_kc_k=\gcd(a_1,\ldots, a_k). \]

%\item \label{item: Dirichlet} \textnormal{(Dirichlet's prime number theorem)\textbf{.}}
%Let $a_1\ldots, a_k$ be  integers such that $\gcd(a_1,\ldots, a_k)=1$.
%Then there are infinitely many primes that are of the form
%\[a_1+ a_2c_2+ \ldots + a_kc_k,  \]
%where $c_2,\ldots, c_k$ are  integers. \qed
Let $a$ and $b$ be  two coprime positive integers.
Then there are infinitely many positive integers $k$
such that 
$a+ k  b$ is a prime number. \qed
%\end{enumerate}
\end{theorem}

%We will also use the following variant of Bez\'out's identity.

%More precisely, we will use the following corollary of Theorem~\ref{theorem: Bezout Dirichlet}.
%
%\begin{lemma}\label{lemma: bezout extension}
%Let $a_1,\ldots, a_{k}$ be positive integers.
%Then there exists integers
%$c_1,\ldots c_k$ such that:
%\begin{itemize}
%\item $ a_1c_1+\ldots+ a_kc_k=\gcd(a_1,\ldots, a_k)$; and
%\item $c_1$ is a prime number. 
%\end{itemize}
%\end{lemma}
%\begin{proof}
%By Theorem~\ref{theorem: Bezout Dirichlet}\eqref{item: Bezout},
%there exists
%integers $c_1',\ldots, c_k'$ such that 
%$ a_1c_1'+\ldots+ a_kc_k'=\gcd(a_1,\ldots, a_k)$.
%Write $d:=\gcd(a_1,\ldots, a_k)$.
% Note that we have:
% \begin{align*}
% a_1\left(c_1'+\sum_{i=2}^{k}\frac{a_i}{d} \, b_i\right) +\sum_{i=2}^k a_i\left(c_i'-\frac{a_1}{d}\, b_i\right) =  \gcd(a_1,\ldots, a_k),
% \end{align*}
%for any integers $b_2,\ldots, b_k$.
%Also note that 
% the sum $ a_1c_1'+\ldots+ a_kc_k'$ is divisible by
% $d\, \gcd(c_1', a_2/d, \ldots, a_k/d)$, 
%and hence we conclude that $\gcd(c_1', a_2/d, \ldots, a_k/d)=1$.
%
%Let $c_1,\ldots, c_k$ be given by:
%\[ c_1:=c_1'+\sum_{i=2}^{k}\frac{a_i}{d} \, b_i; \qquad c_i:=c_i'-\frac{a_1}{d}\, b_i \quad (i\geq 2). \]
%By Theorem~\ref{theorem: Bezout Dirichlet}\eqref{item: Dirichlet},
%there exists integers $b_2,\ldots, b_k$ such that
%$c_1$ is prime.
%Hence $c_1,\ldots, c_k$ satisfies the  condition in the lemma, as desired.
%\end{proof}

\section{Proof of Theorem~\ref{theorem: combinatorial}}
\label{subsection: proof of main theorem 1}

%We start with the case when $I$ is the empty set.
%Note that  $\Fc^{\varnothing}=\{ \varnothing \}$ and $\Nc^{\varnothing}=\{ \{0\} \}$, and the bijection $\psi^{\varnothing}$ is the unique bijection between these two sets.
%
%
%
%
%We now consider the case when $I$ is  a non-empty subset of  $\{1,\ldots,m\}$.

%We assume throughout this section that $q$ and $n$ are two coprime positive integers,
%and that $q$ is a prime power for the first half of this section.

In this section, we present a  proof of Theorem~\ref{theorem: combinatorial}, starting with the case  that $q$ is a prime power.

Let $m$ and  $S_i$  be as in Definition~\ref{definition: S_i}, and $P_i$ be as in Definition~\ref{definition: P_i}.
For any function $f:\Z_n \to \{0,1, \ldots, q-1\}$,
the  \emph{level set} $L_{q-1}(f)$ of $f$ at $q-1$ is the set  $\{z \in \Z_n \mid f(z)=q-1 \}$.
%We will use the following sets to partition $\Fc$ and $\Nc$.
%\renewcommand{\qedsymbol}{$\triangle$}
\begin{definition}\label{definition: Sc^I and Nc^I}
Let $q$ be a prime power.
For any $I \subseteq \{1,\ldots,m\}$,
the sets   $\Nc_I$ and $\Fc_I$ are given by
\begin{align*}
\Nc_I&:=\left\{  \{\alpha,X\alpha,\ldots, X^{n-1}\alpha  \} \in \Nc \  \bigg | \   P_i \text{ divides } \alpha  \text{ iff } i  \notin I\right \},\\
\Fc_I&:=\{ f \in \Fc \ \mid \   L_{q-1}(f) \cap S_i = S_i  \text{ iff } i \notin I \}.    \tag*{\qedhere} 
\end{align*} 
\end{definition}
 By definition $\{ \Nc_I \}_{I \subseteq \{1,\ldots,m\}}$ 
and $\{ \Fc_I \}_{I \subseteq \{1,\ldots,m\}}$ form a partition of $\Nc$ and $\Fc$, respectively.

\begin{example}\label{example: 1}
Continuing from Example~\ref{example: 0}, 
%We make the following choices of $s_1$ and $s_2$:
%\begin{align*}
%& s_1=0, \quad S_1=\{ 0 \}; \quad \text{ and } \quad  s_2=1, \quad S_2=\{1,2\}.
%\end{align*}
%Using this choice of $t_i$'s, the set $\Z_5$ is then given by:
%\[\Z_5=\{0,1,2,4,8\}.   \]
%
%We represent a function $f:\Z_3 \to \{0,1\}$ 
%as the set   $\{ z \in \Z_3 \mid f(z)=1 \}$.
%The  sets $\Fc$ and $\Nc$ are then  given by:
%\begin{align*}
%\Nc=&\{ \{0\}, \{1,X,X^2 \}, \{1+X, X+X^2,1+X^2  \}, \{1+X+X^2   \}\};\\
%\Fc=&\{ \varnothing,  \{0\},  \{1,2\}, \{0,1,2\} \}.
%\end{align*}
%%We make the following choice for $t_i$ and $T_i$:
%%\begin{align*}
%%& t_1=0, \quad T_1=\{ 0 \}; \\
%%&t_2=1, \quad T_2=\{1,2,4,8\}= \{1,2,3,4\} \quad \text{mod } 5.
%%\end{align*}
%The elements  $P_i\in \Qc$ are then given by
%\[P_1=1+X; \qquad P_2=1+X+X^2. \]
the sets $\Fc_I$ and $\Nc_I$ from Definition~\ref{definition: Sc^I and Nc^I} are given by
\begin{alignat*}{2}
&  \Nc_\varnothing=\{\{0\}\}, \qquad &&\Fc_\varnothing=\{\{0,1,2 \} \};\\
&\Nc_{\{1\}}=\{ \{ 1+X+X^2  \}\}, \qquad  &&\Fc_{\{1\}}=\{ \{1,2  \} \}; \\
&\Nc_{\{2\}}=\{ \{1+X,X+X^2,1+X^2  \}  \}, \qquad &&\Fc_{\{2\}}=\{ \{0 \} \};\\
&\Nc_{\{1,2\}}=\{   \{1,X,X^2  \}  \}, \qquad &&\Fc_{\{1,2\}}=\{ \varnothing \}. \tag*{\qedhere} 
\end{alignat*}
\end{example}

We now show that $\Nc_I$ and $\Fc_I$ have the same cardinality for any $I \subseteq \{1,\ldots,m\}$.

Let $s_i$ and $\ell_i$  be as in Definition~\ref{definition: S_i}.

%We identify the following two subsets:
%\begin{align}
%\begin{split}
%\{ \alpha \in Q \mid P_i \text{ divides } \alpha \text{ iff } i \notin I  \} & \to \prod_{i \in I} G_i \\
%\alpha & \mapsto (\alpha_i)_{i \in I}.
%\end{split}
%\end{align}

%Write $d_i:=\gcd(s_i,n)$ for any $i \in \{1,\ldots,m\}$,
%and $d_I:=\gcd(n, \gcd(s_i \mid i \in I))$ for any $I \subseteq \{1,\ldots,m\}$.
\begin{lemma}\label{lemma: size of N_I}
Let $q$ be a prime power,  let $n$ be a positive integer coprime to $q$, and let $I \subseteq \{1,\ldots,m\}$.
Then 
\[  |\Nc_I|=  \frac{\gcd(n, \gcd(s_i)_{i \in I})}{n} \, \prod_{i \in I} (q^{\ell_i}-1). \]
\end{lemma}
\begin{proof}
%First note that 
% $\Nc_{\varnothing}=\{ \{0\}\}$, and in particular $|\Nc_{\varnothing}|=1$.
%This proves the lemma for when $I=\varnothing$.

%We now prove the lemma for when $I$ is a 
% nonempty subset of $\{1,\ldots,m\}$. 
Recall the definition of $G_i$ from Definition~\ref{definition: P_i} and the definition of $\alpha_i$ and $X_i$ from Definition~\ref{definition: alpha_i}.
%Since 
%$\Qc/P_i\Qc$ is a finite field (Lemma~\ref{lemma: finite field}\eqref{item: finite field 1}), 
%for any $\alpha \in \Qc$ we have   $\alpha_i$ is contained in $G_i$ if and only if 
% $P_i$ does not divide $\alpha$.
In particular, if  $\alpha$ is an element of $\Qc$ that is not divisible by $P_i$, then $\alpha_i$ is contained in $G_i$.
Consider the map \begin{align*}
\begin{split}
\xi:  \{ \alpha \in \Qc \mid P_i \text{ divides } \alpha \text{ iff } i \notin I  \} & \to \prod_{i \in I} G_i \\
\alpha & \mapsto (\alpha_i)_{i \in I}.
\end{split}
\end{align*}
%The image of $\xi$ is contained in $\prod_{i \in I} G_i$
%by the definition of $G_i$.
%Therefore, the
The 
map  
$\xi$ 
is a bijection by Theorem~\ref{theorem: Chinese remainder theorem}\eqref{item: CRT2}.
%Since,
%we have $\alpha_i$ is contained in $G_i$ if and only if 
%$P_i$ does not divide $\alpha$.
%It then follows from T that $\xi$ is a bijection.

%Note that, for any $\alpha$ that is not divisible by $P_i$,
%we have $\alpha_i$ is a nonzero element of $\Qc/P_i\Qc$.
%Together with the fact that ,
%this implies that $\alpha_i$ is contained in $G_i$.
%Hence we conclude that 
%the image of $\xi$ is contained in $\prod_{i \in I} G_i$.
%% $\alpha_i$ is a nonzero element of $\Qc/P_i\Qc$ 
%%if  $\alpha$ is not divisible by $P_i$.
%%Hence  
%% by Lemma~\ref{lemma: finite field}\eqref{item: finite field 1}-\eqref{item: finite field 2}  we have   $\alpha_i$ is contained in $G_i$. 
%Also note that that $\xi$ is a bijection by .

Denote by $C_I$  the cyclic subgroup of 
$\prod_{i \in I} G_i$ generated by $(X_i)_{i \in I}$.
Note that  
   $\Nc_I$ is in bijection with cosets  of $C_I$ in $\prod_{i \in I} G_i$ by the map $\xi$.
 Hence we have
 \begin{equation}\label{equation: Nc} |\Nc_I|= \left|  \prod_{i \in I} G_i \Big / C_I \right|=\frac{1}{|C_I|} \, \prod_{i \in I} |G_i|.
 \end{equation}
On the other hand, we also have \begin{align*}
  |G_i|&=q^{\ell_i}-1 \qquad \text{(by Lemma~\ref{lemma: finite field}\eqref{item: finite field 2})}; \\
  |C_I|&=\min\{k>0 \mid (X_i)^k \text{ is the identity element of $G_i$ for all } i \in I  \}\\
  &=\lcm\left(1, \left(\frac{n}{\gcd(n,s_i)}\right)_{i \in I} \right)  \qquad \text{(by Lemma~\ref{lemma: finite field}\eqref{item: finite field 3})}\\
  &=\frac{n}{\gcd(n,\gcd(n,s_i)_{i \in I})}.
%   |C_I|&=\min \{ k \mid   \}\begin{cases}1 & \text{if } I=\varnothing;\\
%      \lcm \left(\frac{n}{\gcd(n,s_i)}\right)_{i \in I}& \text{if } I\neq \varnothing \qquad \text{(by Lemma~\ref{lemma: finite field}\eqref{item: finite field 3})}.
%         \end{cases}\\
%         &=\frac{n}{\gcd(n,\gcd(n,s_i)_{i \in I})}.
 \end{align*}
 The conclusion of the lemma now follows from \eqref{equation: Nc}.
\end{proof}

\begin{lemma}\label{lemma: size of F_I}
Let $q$ and $n$ be two coprime positive integers,
 and let $I \subseteq \{1,\ldots,m\}$.
Then 
\[  |\Fc_I|=  \frac{\gcd(n, \gcd(s_i)_{i \in I})}{n} \, \prod_{i \in I} (q^{\ell_i}-1). \]
\end{lemma}
\begin{proof}
%First note that $\Fc_{\varnothing}$ consists of one function, namely the function that maps $\Z_n$ to $0$.
%Hence $|\Fc_{\varnothing}|=1$, and this proves the lemma for when $I=\varnothing$.

%We now prove the lemma for when $I$ is a nonempty subset of $\{1,\ldots,m\}$.
Let $\Ec_I$ denote the set
\[\Ec_I:=\{ f: \Z_n \to \{0,1,\ldots, q-1\} \ \mid \   L_{q-1}(f) \cap S_i = S_i  \text{ iff } i \notin I \}. \]
Let $\eta_I: \Ec_I\to \prod_{i \in I} \Z_{q^{\ell_i}-1}$ be the map defined by
\begin{align*}
f &\mapsto \left( \sum_{j=0}^{\ell_i-1} q^{j} f(q^{j} s_i) \mod q^{\ell_i}-1\right)_{i \in I}.
\end{align*}
The  map $\eta_I$ is  surjective  by the definition of $\Ec_I$.

%Note that, for any $i \in I$, the sum $\sum_{j=0}^{\ell_i-1} q^{j} f(q^{j} s_i$ is strictly less than $q^{\ell_i}-1$ 

%For any $i \in I$, let $\eta_i: \Ec_I\to \Z_{q^{\ell_i}-1}$ be the map
%\begin{align*}
%f &\mapsto \sum_{j=0}^{\ell_i-1} q^{j} f(q^{j} s_i).
%\end{align*}
%%Note that the image of $\eta_i$ is contained in $\Z_{q^{\ell_i}-1}$
%%since 
%%$(f(s_i),\ldots, f(q^{\ell_i-1}s_i))$ is not equal to  $(q-1,\ldots, q-1)$ for any $f \in \Ec_I$.
%Note  that $\eta_i$ is surjective by the definition of $\eta_i$.
%Recall the definition of $S_i$ from   Definition~\ref{definition: S_i}.

Let $f$ be any function in $\Ec_I$.
For any $i \in I$,
the sum $\sum_{j=0}^{\ell_i-1} q^{j} f(q^{j} s_i)$ 
is strictly less than $q^{\ell_i}-1$  since $L_{q-1}(f) \cap S_i \neq S_i$.
This implies  that the $i$-th coordinate of $\eta_I(f)$
 determines  $f(s_i),\ldots, f(q^{\ell_i-1}s_i)$ for any $i \in I$. 
Furthermore, we have $(f(s_i),\ldots, f(q^{\ell_i-1}s_i))= (q-1,\ldots, q-1)$
for any $i \notin I$ by the definition of $\Ec_I$.
% such that 
%$(f(s_i),\ldots, f(q^{\ell_i-1}s_i))$
%is neither $(0,\ldots,0)$ nor $(q-1,\ldots,q-1)$,
%the image of $f$ under $\eta_i$ determines  $f(s_i),\ldots, f(q^{\ell_i-1}s_i)$. 
%Also note that, for any $f \in \Fc_I$, 
% we have by Definition~\ref{definition: Sc^I and Nc^I} that
%  $(f(s_i),\ldots, f(q^{\ell_i-1}s_i))$ is not equal to  $(q-1,\ldots, q-1)$.
%  Hence for any $f \in \Fc_I$ we have
% $\eta_i(f)=0$ if and only if   $(f(s_i),\ldots, f(q^{\ell_i-1}s_i))=(0,\ldots,0)$.
%  %It then follows from the definition of $\Fc_I$ (Definition~\ref{definition: Sc^I and Nc^I}) that the map   
  Therefore, we conclude that $\eta_I$ is an injective map.

%\begin{align*}
%\eta_I: \Ec_I \to \prod_{i \in I} \Z_{q^{\ell_i}-1}\\
%f\mapsto (\eta_i(f))_{i \in I},
%\end{align*}
%is injective.
%Since  $\eta_i$ is a surjection for any $i \in I$,
%we then conclude that $\eta_I$ is a bijection.

Let $\zeta_I$ be the map defined by
\begin{align*}
\begin{split}
\zeta_I:\prod_{i \in I} \Z_{q^{\ell_i}-1} &\to \Z_n\\
(z_i)_{i \in I} & \mapsto \sum_{i \in I} s_iz_i \mod n.
\end{split}
\end{align*}
The map $\zeta_I$ is a well defined group homomorphism as $n$ divides $s_i(q^{\ell_i}-1)$ for all $i \in \{1,\ldots,m\}$ by Definition~\ref{definition: S_i}.
Furthermore, by the definition of $\gcd$,  
 the image of $\zeta_I$  is  $\gcd(n,\gcd(s_i)_{i \in I})\Z_n$.

Now note that, for any $f \in \Ec_I$,
\begin{align*}
\sum_{z \in \Z_n} zf(z)=& \sum_{i \in I} \sum_{j=0}^{\ell_i-1} q^{j}  s_i  f({q^{j}s_i})+ \sum_{i \notin I} (q^{\ell_i}-1)s_i \\=& \sum_{i \in I} s_i \sum_{j=0}^{\ell_i-1} q^{j}   f({q^{j}s_i}) \quad \text{(mod $n$)}\\
=&\zeta_I(\eta_I(f)).
\end{align*} 
Since $\eta_I$ is a bijection, 
it then follows from the definition of $\Fc_I$ (Definition~\ref{definition: Sc^I and Nc^I}) that  the kernel of $\zeta_I$ is equal to $\eta_I(\Fc_I)$.
 
Combining all those observations, we conclude that
\begin{align*}
|\Fc_I|=&|\eta_I(\Fc_I)|=|\ker(\zeta_I)|=
\frac{|\prod_{i \in I} \Z_{q^{\ell_i}-1}|}{|\gcd(n,\gcd(s_i)_{i \in I})\Z_n|}\\
=&\frac{\gcd(n, \gcd(s_i)_{i \in I})}{n} \, \prod_{i \in I} (q^{\ell_i}-1),
\end{align*}
as desired.
\end{proof}

We now complete the  proof of Theorem~\ref{theorem: combinatorial}.
\begin{proof}[Proof of Theorem~\ref{theorem: combinatorial}]

Fix an arbitrary positive integer $n$.
Let $r \in \{0,\ldots,n-1\}$ be such that $\gcd(n,r)=1$.
Let $x$ be a variable, and let $q=xn+r$ throughout this proof.
Note that $\gcd(n,r)=\gcd(n,q)=1$. 
%Let $x$ be a positive integer and  $r$ be a positive integer copri \in \{1,\ldots, n-1\}$
%be an integer coprime to $n$ such that 
% $q=xn+r$.
%Fix the values of  $n$ and $r$ throughout this proof.
%Let   $n$ and $r$ be fixed constants   and let $x$ be a variable throughout this proof.

Since the integers   $m$, $s_i$, and $\ell_i$  from Definition~\ref{definition: S_i}  depend only on $n$ and $r$, we have  the function 
\begin{equation}\label{equation: polynomial}
 x \mapsto \sum_{I \subseteq \{1,\ldots,m\}} \frac{\gcd(n, \gcd(s_i)_{i \in I})}{n}   \, \prod_{i \in I} ((xn+r)^{\ell_i}-1)
\end{equation}
is a polynomial of $x$.

By Lemma~\ref{lemma: size of N_I}, Lemma~\ref{lemma: size of F_I}, and the fact that $\{\Nc_I\}_{I \subseteq \{1,\ldots,m\}}$ and $\{\Fc_I\}_{I \subseteq \{1,\ldots,m\}}$ form a  partition of $\Nc$ and $\Fc$ respectively,
we have that
 $|\Nc|$ and $|\Fc|$ are equal to  the polynomial in \eqref{equation: polynomial} when $q=xn+r$ is a prime power.
Since $\gcd(n, r) = 1$, we have by Theorem~\ref{theorem: Bezout Dirichlet} that
   there are infinitely many  positive integers $x$ for which $xn+r$
 is a prime.
Hence it suffices to show that $|\Nc|$ and $|\Fc|$ are polynomials of $x$.

For any $i \in \{1,\ldots, n\}$, let $\text{col}(i)$ be the number of necklaces of length $n$ with colors chosen from $\{0,\ldots i-1\}$, and such that all $i$ colors are  used.
Then
\begin{align*}
|\Nc|=\sum_{i=1}^n \text{col}(i)\, \binom{q}{i}=\sum_{i=1}^n \frac{\text{col}(i)}{i!}\, \prod_{j=0}^{i-1}(xn+r-j).
\end{align*}
This shows that  $|\Nc|$ is a polynomial of $x$.

Let $V$ denote the set 
%\[ V:=\{ (r_0,\ldots, r_{n-1}) \in \{1,\ldots,n\}^{n} \ \mid \ r_0+\ldots +r_{n-1}=0 \  (\text{mod } n) \}. \]
\[ V:=\left \{ R \in \{0,\ldots,n-1\}^{\Z_n} \ \biggl | \ \sum_{z\in \Z_n} zR_z=0 \  (\text{mod } n) \right \}. \]
%For any $r_0,\ldots, r_{n-1}\in \{1,\ldots n\}$, let $\text{sol}(r_0,\ldots, r_{n-1})$ be given by:
%\[\text{sol}(r_0,\ldots, r_{n-1}):=\begin{cases} 1 & \text{if } r_0+\ldots +r_{n-1}=0 \  (\text{mod } n);\\
%0 & \text{otherwise}.\end{cases}   \]
We then have 
\begin{align*}
|\Fc|=& \sum_{R \in V } \{ f: \Z_n \to \{0,1, \ldots,q-1\}   \ \mid \      f(z)= R_z   \text{ (mod $n$) for all }z \in \Z_n  \}    \\
=& \sum_{R \in V } \prod_{z\in \Z_n} \left| \{  k\geq 0 \ | \  kn+R_z < q   \}\right|  \\
=&\sum_{R \in V }  
(x+1)^{|\{z\in \Z_n  \, \mid\,  R_z <r\}|} \, x^{|\{z \in \Z_n \, \mid\,  R_z \geq r\}|}.
%\, \prod_{\substack{i : 0\leq i<n,\\  R_z< r}} (x+1) \, \prod_{\substack{0\leq i<n,\\  r_i\geq r}} x.
\end{align*}
This shows that $|\Fc|$ is a polynomial of $x$.
This completes the proof.
\end{proof}

\section{Proof of Theorem~\ref{theorem: bijection}}
\label{subsection: proof of main theorem 2}

In this section, we present a proof of Theorem~\ref{theorem: bijection}.
Throughout this section,  $q$ is a prime power  and $n$ is a positive integer that is coprime to $q$.

Let  $\Qc$ be as defined in Section~\ref{s. notation}, and let
$\Ec$ be the set of all functions from $\Z_n$ to $\{0,1, \ldots, q-1\}$.
%Recall that $\Nc$ and $\Fc$ are given by:
%\begin{align*}
%\Nc=&\{    \mid \alpha \in \Qc \};
%\Fc=&\{f \mid \sum_{}   \}
%\end{align*}
Suppose that there exists a map $\psi: \Qc \to \Ec$ 
that satisfies the following conditions:
\begin{enumerate}
[{
label=\textnormal{({C\arabic*})}, labelindent=0pt,
ref={C\arabic*}}]
\item \label{item: C1} The map $\psi$ is a bijection from $\Qc$ to $\Ec$; and
\item \label{item: C2} For any $\alpha \in \Qc$ there exists a unique 
 $\beta \in \{\alpha, X\alpha, \ldots, X^{n-1}\alpha\}$
such that 
 $\psi(\beta)$ is contained in $\Fc$.
\end{enumerate}
We could then  define the map $\widehat{\psi}: \Nc \to \Fc$  by
\[  \{\alpha, X\alpha,\ldots,  X^{n-1}\alpha\} \mapsto \psi(\beta). \]
It  would follow that $\widehat{\psi}$ is a bijection between $\Nc$ and $\Fc$,
which 
would   prove Theorem~\ref{theorem: bijection}.
In this section, we will construct a map $\psi:\Qc \to \Ec$ that satisfies \eqref{item: C1} and \eqref{item: C2}.

Recall the definition of  $m$, $s_i$, and $\ell_i$   from Definition~\ref{definition: S_i},
%and let $I$ be an arbitrary subset of $\{1,\ldots,m\}.
 the definition of $G_i$ from  
Definition~\ref{definition: P_i},
and the definition of $X_i$ from Definition~\ref{definition: alpha_i}.

Let $i \in \{1,\ldots,m\}$.
Since $G_i$ is  a cyclic group of order $q^{\ell_i}-1$ (Lemma~\ref{lemma: finite field}\eqref{item: finite field 2}) and $X_i$ is an element of $G_i$ with order $\frac{n}{\gcd(n,s_i)}$ (Lemma~\ref{lemma: finite field}\eqref{item: finite field 3}), the group $G_i$ contains a group generator such that $X_i$ is  ${\frac{(q^{\ell_i}-1)\gcd(n,s_i)}{n}}$-th power of this generator.
\begin{definition}\label{definition: g_i}
For any $i \in  \{1,\ldots,m\}$,
let $g_i$ be a group generator of  $G_i$
such that $X_i$ is the ${\frac{(q^{\ell_i}-1)\gcd(n,s_i)}{n}}$-th power of $g_i$. 
\end{definition}

%Note that $g_i$ exists as $G_i$ is a cyclic group of order $q^{\ell_i}-1$ (Lemma~\ref{lemma: finite field}\eqref{item: finite field 2})
%and $X_i$ is an element of $G_i$ with order $\frac{n}{\gcd(n,s_i)}$ (Lemma~\ref{lemma: finite field}\eqref{item: finite field 3}).

Recall the definition of $P_i$ from Definition~\ref{definition: P_i} and the definition of  $\alpha_i$ from Definition~\ref{definition: alpha_i}.
\begin{definition}[Discrete logarithm]\label{definition: discrete logarithm}
Let  $i\in \{1,\ldots,m\}$, and let $\alpha$ be an element of  $\Qc$ not divisible by  $P_i$.
The \emph{discrete logarithm} $\log_{g_i}(\alpha)$ 
is  the smallest non-negative integer $k$ such that $\alpha_i=g_i^k$ in $G_i$.
\end{definition}
 By Lemma~\ref{lemma: finite field}\eqref{item: finite field 2}, the integer $\log_{g_i}(\alpha)$ is contained in $\{0,\ldots, q^{\ell_i}-2\}$. 

\begin{definition}\label{definition: a_i b_i}
Let  $i\in \{1,\ldots,m\}$, and let $\alpha$ be an element of  $\Qc$ not divisible by  $P_i$.
We denote by  $a_i(\alpha)$ and $b_i(\alpha)$ the quotient and the remainder of the division of $\log_{g_i}(\alpha)$ by $\frac{(q^{\ell_i}-1) \gcd(n,s_i)}{n}$, respectively.
%That is $a_(\alpha)$ and $b_i(\alpha)$ are the unique integers satisfying
%%We denote by $a_i(\alpha)$ and $b_i(\alpha)$ the unique integers satisfying
%\begin{itemize}
%\item $0\leq a_i(\alpha) <\frac{n}{\gcd(n,s_i)}$ and $0\leq b_i(\alpha) <\frac{(q^{\ell_i}-1) \gcd(n,s_i)}{n}$; and
%\item 
%$\displaystyle \log_{g_i}(\alpha)=  {a_i(\alpha)  \frac{(q^{\ell_i}-1) \gcd(n,s_i)}{n} +b_i(\alpha) }.$ \qedhere
%\end{itemize}
\end{definition}
In particular, the nonnegative integers $a_i(\alpha)$ is strictly less than  $ \frac{n}{\gcd(n,s_i)}$  and $b_i(\alpha)$ is strictly less than  $\frac{(q^{\ell_i}-1) \gcd(n,s_i)}{n}$. 
We compute these integers for the case $n=3$ below.

%See Example~\ref{example: 3}
%for an explicit computation of $a_i$'s and $b_i$'s for $n=5$.

\begin{example}\label{example: 2}
Continuing from Example~\ref{example: 1}, we  make the following choices of  $g_1$ and $g_2$ that satisfy the condition in Definition~\ref{definition: g_i}:
\begin{align*}
 g_1=1 \text{ mod } 1+X \quad \text{ and } \quad    g_2=X \text{ mod } 1+X+X^2.
\end{align*}
Note that $\frac{(q^{\ell_i}-1) \gcd(n,s_i)}{n}=1$ for $i \in \{1,2\}$;
we remark that this equality is special  to this example and is false for large values of $n$ and $q$.

The following is the value of $\log_{g_1}(\alpha)$,  $a_1(\alpha)$ and $b_1(\alpha)$ for different $\alpha$'s:
\begin{itemize}
\item $\text{If } \alpha=1  \ \ \text{(mod }1+X),  \text{ then }$
\[ \log_{g_1}(\alpha)=0; \qquad   a_1(\alpha)=0; \qquad  b_1(\alpha)=0.\]
\end{itemize}

The following is the value of $\log_{g_2}(\alpha)$,  $a_2(\alpha)$ and $b_2(\alpha)$ for different $\alpha$'s:

\begin{itemize}
\item $\text{If } \alpha=1  \ \ \text{(mod }1+X+X^2),  \text{ then }$
\[ \log_{g_2}(\alpha)=0; \qquad   a_2(\alpha)=0; \qquad  b_2(\alpha)=0.\]

\item $\text{If } \alpha=X  \ \ \text{(mod }1+X+X^2),  \text{ then }$
\[ \log_{g_2}(\alpha)=1; \qquad   a_2(\alpha)=1; \qquad  b_2(\alpha)=0.\]

\item $\text{If } \alpha=1+X  \ \ \text{(mod }1+X+X^2),  \text{ then }$
\[ \log_{g_2}(\alpha)=2; \qquad   a_2(\alpha)=2; \qquad  b_2(\alpha)=0. \tag*{\qedhere}\]
\end{itemize}

\end{example}

\begin{lemma}\label{lemma: a_i b_i}
Let $q$ be a prime power, let $n$ be a positive integer coprime to $q$, and let $i \in \{1,\ldots,m\}$. Then 
\begin{enumerate}[{
label=\textnormal{({\roman*})}, labelindent=0pt,
ref={\roman*}}]
\item \label{item: a_i b_i 1} $a_i(X)=1$  and $b_i(X)=0$; and  

\item \label{item: a_i b_i 2} For any  $k\geq 0$ and any $\alpha \in \Qc$,
\begin{align*}
\begin{split}
a_i( X^k \alpha) &=  k  +a_i(\alpha) \quad \left(\textnormal{mod } \frac{n}{\gcd(n,s_i)}\right); \ \text{and}\\
 b_i(X^k\alpha)&=b_i(\alpha).
\end{split}
\end{align*}
\end{enumerate}
\end{lemma}
\begin{proof}
Part \eqref{item: a_i b_i 1}  follows directly from Definition~\ref{definition: g_i} and  Definition~\ref{definition: a_i b_i}.

By Definition~\ref{definition: discrete logarithm}, we  have for  any non-negative integer $k$ and any $\alpha \in \Qc$ that
\begin{align*}
 \log_{g_i}(X^k \alpha)=&\log_{g_i}(\alpha)+ k\log_{g_i}(X) \quad   (\text{mod } q^{\ell_i}-1)\\
 =& \left(k  +a_i(\alpha) \right) \frac{(q^{\ell_i}-1) \gcd(n,s_i)}{n} 
 +b_i(\alpha)  \quad \text{(mod } q^{\ell_i}-1).
\end{align*}
Part \eqref{item: a_i b_i 2} now follows from Definition~\ref{definition: a_i b_i}.
\end{proof}

\begin{definition}\label{definition: phi}
Let $I$ be a  subset of $\{1,\ldots,m\}$.
Let  $\phi_I$ be a  group automorphism of  
    $\prod_{i \in I} \Z_{\frac{n}{\gcd(n,s_i)}}$
  that satisfies
 \begin{equation}\label{equation: definition phi}
 \sum_{i \in I} s_ih_{i,I}= \gcd(n,\gcd(s_i)_{i \in I}) \quad \textnormal{(mod $n$)},
 \end{equation}
  where $h_{i,I}$ is the $i$-th coordinate of $\phi_I(1,\ldots,1)$.
  The function $\phi_I$ exists for any $I \subseteq \{1,\ldots,m\}$ by
Lemma~\ref{lemma: Bezout variant}.
\end{definition}
 We present an explicit example of the function  $\phi_I$ for the case $n=3$ below.

\begin{example}\label{example: 3}
Continuing from Example~\ref{example: 2}, we choose $\phi_I$ to be the identity map on
$\prod_{i \in I} \Z_{\frac{n}{\gcd(n,s_i)}}$ for any $I \subseteq \{1,2\}$.
The map $\phi_I$ satisfies \eqref{equation: definition phi}  by the following computation:
\begin{itemize}
\item When $I=\varnothing$, the condition in \eqref{equation: definition phi} is vacuously true.
\item When $I=\{1\}$, we have
\[ s_1 h_{1,\{1\}}=0\cdot  1=3 \quad \text{(mod 3)}.   \]

\item When $I=\{2\}$, we have
\[ s_2 h_{2,\{2\}}=1 \cdot 1=1 \quad \text{(mod 3)}.   \]

\item When $I=\{1,2\}$,we have
\[ s_1 h_{1,\{1,2\}}+s_2h_{2,\{1,2\}} =0\cdot 1+1\cdot 1= 1 \quad \text{(mod 3)}. \tag*{\qedhere} \]
\end{itemize} 
%
%
%We choose the following $\phi_I$'s for each  subset of $\{1,\ldots,m\}$:
%\begin{itemize}
%\item When $I=\psi$, the function $\phi_{\varnothing}:\{\varnothing \}\to \{\varnothing \} $ is
%the identity map.
%\item When $I=\{1\}$, the function $\phi_{\{1\}}:\Z_{1}\to \Z_1$ is
%the identity map.
%\item When $I=\{2\}$, the function $\phi_{\{2\}}:\Z_3\to \Z_3$ is
%the identity map.
%\item When $I=\{1,2\}$, the function $\phi_{\{1,2\}}:\Z_1 \times \Z_3\to \Z_1 \times \Z_3$ is given by the identity map.
%Note that 
%\[ s_1\phi_1+s_2\phi_2 =1+3= 1 \quad \text{(mod 3)}, \]
%and hence $\phi_{\{1,2\}}$ satisfies the condition in Definition~\ref{definition: phi}. \qedhere
%\end{itemize} 
\end{example}

Recall that $L_{q-1}(f)=\{ z \in \Z_n \mid f(z) =q-1  \}$.
For any  $I\subseteq \{1,\ldots,m\}$, write
\begin{align*}
\Qc_I&:=\left\{  \alpha \in \Qc \  \bigg | \   P_i \text{ divides } \alpha  \text{ iff } i  \notin I\right \};\\
\Ec_I&:=\{ f \in \Ec \ \mid \   L_{q-1}(f) \cap S_i = S_i  \text{ iff } i \notin I \}. \qedhere
\end{align*}
By definition  $\{ \Qc_I \}_{I \subseteq \{1,\ldots,m\}}$ 
and $\{ \Ec_I \}_{I \subseteq \{1,\ldots,m\}}$ form a partition of $\Qc$ and $\Ec$, respectively.

Let  $i \in I$, and let $\alpha$ be any element of $Q_I$.
 We denote by $\phi_{i,I}(\alpha)$  the $i$-th coordinate of $\phi_{I}((a_i(\alpha))_{i \in I})$,
 which corresponds to a nonnegative integer strictly less than  $\frac{n}{\gcd(n,s_i)}$.

Since  $b_i(\alpha)$ is a nonnegative integer  strictly less than $\frac{(q^{\ell_i}-1) \gcd(n,s_i)}{n}$  and $\phi_{i,I}(\alpha)$ is a nonnegative integer strictly less than $\frac{n}{\gcd(n,s_i)}$, we have  
\begin{equation}\label{equation: sum requested by referee} 0\leq  b_{i}(\alpha) \frac{n}{\gcd(n,s_i)} +\phi_{i,I}(\alpha) < q^{\ell_i}-1. \end{equation}
We denote by 
$c_{i,0}(\alpha),\ldots, c_{i,\ell_i-1}(\alpha) \in \{0,\ldots, q-1\}$ the unique  integers that satisfy
\begin{equation}\label{equation: cij}
\sum_{j=0}^{\ell_i-1} c_{i,j}(\alpha) \, q^{j}= b_{i}(\alpha) \frac{n}{\gcd(n,s_i)} +\phi_{i,I}(\alpha). 
\end{equation}
By \eqref{equation: sum requested by referee}, the sequence of integers  $(c_{i,0},\ldots, c_{i,\ell_i-1})$ is well defined
and  is not equal to $(q-1,\ldots,q-1)$.

Let $f_\alpha:\Z_n \to \{0,1, \ldots,q-1\}$  be given by
\begin{equation}\label{definition: function alpha}
f_{\alpha}(q^js_i):= \begin{cases}
q-1 & \text{if } i \notin I;\\
c_{i,j}(\alpha) & \text{if } i \in I.
\end{cases}
\end{equation}
The function  $f_\alpha$ has the property that $L_{q-1}(f_\alpha) \cap S_i$ is a strict subset of $S_i$ for any $i \in I$ since $(c_{i,0},\ldots, c_{i,\ell_i-1})$ is   is not equal to $(q-1,\ldots,q-1)$.
This implies that  $f_{\alpha}$ is contained in $\Ec_I$.

\begin{definition}\label{definition: varphi_I}
Let $I \subseteq \{1,\ldots,m\}$.
We define $\psi_I: \Qc_I \to \Ec_I$ to  be the map that sends $\alpha \in \Qc_I$ to the function $f_\alpha$.
%where $c_{i,0},\ldots, c_{i,\ell_i-1} \in \{0,\ldots q-1\}$ are the integers such that:
%\[
%\sum_{j=0}^{\ell_i-1} c_{i,j} q^{j}= b_{i}(\alpha) \frac{n}{\gcd(n,s_i)} +\phi_{i,I}(\alpha), \]
%where $\phi_{i,I}(\alpha)$ is the $i$-th coordinate of $\phi_{I}((a_i(\alpha))_{i \in I})$.
\end{definition}
%Each function $f_\alpha$ has the property that $L_{q-1}(f) \cap S_i$ is a strict subset of $S_i$ for any $i \in I$ since $(c_{i,0},\ldots, c_{i,\ell_i-1})$ is   is not equal to $(q-1,\ldots,q-1)$.
%This implies that each $f_{\alpha}$ is contained in $\Ec_I$,
%and hence so does the image of 
%the  map $\psi_I$.

\begin{example}\label{example: 4}
Continuing from Example~\ref{example: 3}, 
 we present the image of the function $\psi_I$ for different  $\alpha$'s
(recall that we represent a function 
 $f:\Z_3 \to \{0,1\}$ 
as the set  $\{ z \in \Z_3 \mid f(z)=1 \}$):
\begin{itemize}
\item The case $I=\varnothing$:
When $\alpha=0$, the map  $\psi_{\varnothing}$ sends $\alpha$ to  $\{0,1,2  \}$.
\item The case $I=\{1\}$:
When $\alpha=1+X+X^2$, we have
\[ b_{1}(\alpha)  +\phi_{1,\{1\}}(\alpha)=0+0=0=0 \cdot 2^0.   \]
The map $\psi_{\{1\}}$ then sends $\alpha$ to  $\{1,2  \}$.
\item The case $I=\{2\}$:
\begin{itemize}
\item When $\alpha=1+X$, we have
\[ 3 \, b_{2}(\alpha) +\phi_{2,\{2\}}(\alpha)=3\cdot0+2=2=0\cdot 2^0+1 \cdot 2^1.   \]
The map $\psi_{\{2\}}$ then sends $\alpha$ to  $\{0,2 \}$.

\item When $\alpha=X+X^2$, we have
\[ 3\, b_{2}(\alpha)  +\phi_{2,\{2\}}(\alpha)=0\cdot0+0=0=0\cdot 2^0+0\cdot 2^1.   \]
The map $\psi_{\{2\}}$ then sends $\alpha$ to  $\{0 \}$.

\item When $\alpha=1+X^2$, we have
\[ 3\, b_{2}(\alpha) +\phi_{2,\{2\}}(\alpha)=3\cdot 0 +1=1=1 \cdot  2^0+0 \cdot 2^1.   \]
The map $\psi_{\{2\}}$ then sends $\alpha$ to  $\{0,1  \}$.
\end{itemize}

\item
The case $I=\{1,2\}$:
\begin{itemize}
\item 
 When $\alpha=1$, we have
\begin{align*}
 b_{1}(\alpha)  +\phi_{1,\{1,2\}}(\alpha)&=0+0=0=0 \cdot 2^0; \\
 3\, b_{2}(\alpha) +\phi_{2,\{1,2\}}(\alpha)&=3\cdot0+0=0=0 \cdot 2^0+0\cdot 2^1.   
\end{align*}
The map $\psi_{\{1,2\}}$ then sends $\alpha$ to  $\varnothing$.

\item When $\alpha=X$, we have
\begin{align*}
 b_{1}(\alpha)  +\phi_{1,\{1,2\}}(\alpha)&=0+0=0=0 \cdot  2^0; \\
 3\, b_{2}(\alpha)  +\phi_{2,\{1,2\}}(\alpha)&=3\cdot0+1=1=1 \cdot 2^0+0 \cdot 2^1.   
\end{align*}
The map $\psi_{\{1,2\}}$ then sends $\alpha$ to  $\{1\}$.

\item When $\alpha=X^2$, we have
\begin{align*}
 b_{1}(\alpha)  +\phi_{1,\{1,2\}}(\alpha)&=0+0=0=0 \cdot 2^0; \\
 3\, b_{2}(\alpha)  +\phi_{2,\{1,2\}}(\alpha)&=3\cdot 0+2=2=0\cdot 2^0+1\cdot 2^1.   
\end{align*}
The map $\psi_{\{1,2\}}$ then sends $\alpha$ to the function $\{2  \}$. \qedhere
\end{itemize}
\end{itemize}
\end{example}

\begin{lemma}\label{lemma: bijection}
Let $q$ be a prime power,  let $n$ be a positive integer coprime to $q$, and let $I \subseteq \{1,\ldots,m\}$.
Then 
 the map $\psi_I :\Qc_I \to \Ec_I$ is a bijection.
\end{lemma}

\begin{proof}
Let $\alpha$ and $\alpha'$ be two elements of $\Qc_I$ with the same image under $\psi_I$.
By \eqref{equation: cij}, \eqref{definition: function alpha}, and   the definition of $\psi_I$, we have
\[  b_{i}(\alpha) \frac{n}{\gcd(n,s_i)} +\phi_{i,I}(\alpha)= b_{i}(\alpha') \frac{n}{\gcd(n,s_i)} +\phi_{i,I}(\alpha') \quad \text{ for any $i \in I$}.\]
Since $\phi_{i,I}(\alpha)$ and  $\phi_{i,I}(\alpha')$
are both nonnegative integers strictly less than $\frac{n}{\gcd(n,s_i)}$, and  the equation above then implies that
\[ \phi_{i,I}(\alpha)=\phi_{i,I}(\alpha') \qquad \text{ and } \qquad b_i(\alpha)=b_i(\alpha') \quad \text{ for any $i \in I$}.  \]
Since $\phi_I$ is chosen to be a bijection by Definition~\ref{definition: phi},
we conclude that
\[ a_i(\alpha)=a_i(\alpha') \quad \text{ and } \quad b_i(\alpha)=b_i(\alpha')  \quad \text{ for any $i \in I$}.  \]
It then follows from Definition~\ref{definition: g_i} and Definition~\ref{definition: a_i b_i}
that
\[  \alpha= \alpha' \quad \left( \textnormal{mod } P_i \right) \quad \text{ for any $i \in I$}.   \]
On the other hand, by the definition of $\Qc_I$, we have  
\[ \alpha= 0=\alpha' \quad \left( \textnormal{mod } P_i \right) \quad \text{ for any $i \notin I$}.   \]
By Theorem~\ref{theorem: Chinese remainder theorem}\eqref{item: CRT2}, 
we then conclude that 
$\alpha=\alpha'$.
This proves the injectivity of $\psi_I$.

Let $f$ be an arbitrary element of $\Ec_I$.
For any $i \in I$,
let $b_i$ and $\phi_{i,I}$ be the quotient and the remainder of the division of 
the sum $\sum_{j=0}^{\ell_i-1} q^j  f(q^j s_i)$ by $\frac{n}{\gcd(n,s_i)}$.
The  sum $\sum_{j=0}^{\ell_i-1} q^j  f(q^j s_i)$ is a nonnegative integer strictly  less than $q^{\ell_i}-1$
by the assumption that $L_{q-1}(f)\cap S_i \neq S_i$.
This implies that $b_i$ and $\phi_{i,I}$ satisfy the  inequalities
  $0\leq b_i <\frac{(q^{\ell_i}-1) \gcd(n,s_i)}{n}$ and $0\leq \phi_{i,I} < \frac{n}{\gcd(n,s_i)}$.

%let $b_i$ and $\phi_{i,I}$ be the unique integers satisfying
%\begin{itemize}
%\item $ \displaystyle \sum_{j=0}^{\ell_i-1} q^j  f(q^j s_i) = b_i  \frac{n}{\gcd(n,s_i)} + \phi_{i,I}$; and
%\item $0\leq \phi_{i,I} < \frac{n}{\gcd(n,s_i)}$ and $0\leq b_i <\frac{(q^{\ell_i}-1) \gcd(n,s_i)}{n}$. 
%\end{itemize}
%Note that $b_i$ and $\phi_{i,I}$ exist because the sum $\sum_{j=0}^{\ell_i-1} q^j  f(q^j s_i)$ is strictly less than $q^{\ell_i}-1$
%by the assumption that $L_{q-1}(f)\cap S_i \neq S_i$.

Write $(a_i)_{i \in I}:=\phi_I^{-1}((\phi_{i,I})_{i \in I})$.
By Theorem~\ref{theorem: Chinese remainder theorem}\eqref{item: CRT2}
there exists a unique $\alpha \in \Qc$ that satisfies the following equations:
\begin{align*}
  \log_{g_i}(\alpha)&=   { a_i\frac{(q^{\ell_i}-1) \gcd(n,s_i)}{n} +b_i} \qquad (\text{for } i \in I);  \\
  \alpha&= 0 \quad \left( \textnormal{mod } P_i\right) \qquad (\text{for }i\notin I).
\end{align*}
The element $\alpha$ is contained in $\Qc_I$ as $\alpha$ is divisible by $P_i$ if and only if $i\notin I$. 
Furthermore, the map $\psi_I$ maps  $\alpha$ to $f$, as the construction above mirrors the construction of $\psi_I$ with steps taken in the reverse order.
%where $(a_i)_{i \in I}:=\phi_I^{-1}((\phi_{i,I})_{i \in I})$.
%Note that  $\alpha$ exists and is unique by Theorem~\ref{theorem: Chinese remainder theorem}\eqref{item: CRT2}.
%Also note that  $\alpha$ is contained in $\Qc_I$ by definition.
%Finally, note that $\psi_I$ maps $\alpha$ to $f$ by construction.
This proves the surjectivity of $\psi_I$.
\end{proof}

\begin{lemma}\label{lemma: function sum}
Let $q$ be a prime power,  let $n$ be a positive integer coprime to $q$, and let $I \subseteq \{1,\ldots,m\}$.
Then, for any $\alpha \in \Qc_I$,
\begin{enumerate}
[{label=\textnormal{({\roman*})}, labelindent=0pt,
ref={\roman*}}]
\item \label{item: subset sum 1} $\displaystyle \sum_{z \in \Z_n} zf_\alpha(z)= \sum_{i \in I} s_i \phi_{i,I}(\alpha) \quad \textnormal{(mod $n$)}$; and
\item \label{item: subset sum 2} There exists unique $\beta\in \{\alpha, X \alpha, \ldots, X^{n-1}\alpha\}$ 
such that  $\psi_I(\beta)$ is contained in $\Fc$.
\end{enumerate}
\end{lemma}
\begin{proof}
%\begin{enumerate}
%[{label=\textnormal{({\roman*})}, labelindent=0pt,
%ref={\roman*}}]
%\item 
We start with proving part~\eqref{item: a_i b_i 1}.
 We have
\begin{align*}
\sum_{z \in \Z_n} z f_\alpha(z)=& \sum_{i \in I} \sum_{j=0}^{\ell_i-1} q^{j}  s_i c_{i,j}(\alpha)+ \sum_{i \notin I}  \sum_{j=0}^{\ell_i-1}  q^j s_i (q-1) \qquad \text{(by \eqref{definition: function alpha})}
\\
=& \sum_{i \in I} \sum_{j=0}^{\ell_i-1} q^{j}  s_i  c_{i,j}(\alpha)+ \sum_{i \notin I} (q^{\ell_i}-1)s_i\\
=&\sum_{i \in I} \sum_{j=0}^{\ell_i-1} q^{j}  s_i c_{i,j}(\alpha)  \quad \text{(mod $n$)}  \qquad \text{(by Definition~\ref{definition: S_i})}\\
=&\sum_{i \in I}   s_i \left( b_{i}(\alpha) \frac{n}{\gcd(n,s_i)} +\phi_{i,I}(\alpha) \right) \quad \text{(mod $n$)} \qquad \text{(by \eqref{equation: cij})}\\
=&\sum_{i \in I} s_i \phi_{i,I}(\alpha) \quad \text{(mod $n$)}.
\end{align*}
This proves part~\eqref{item: a_i b_i 1}.

We now prove part \eqref{item: a_i b_i 2}.
We have
\begin{align*}
|\{\alpha, X \alpha, \ldots, X^{n-1}\alpha\}|=\min\{k > 0 \mid X^{k}\alpha=\alpha  \}
=\lcm(1,(\text{order of $X$ in $G_i$})_{i \in I}),
\end{align*}
where the last equality is a consequence of Theorem~\ref{theorem: Chinese remainder theorem}\eqref{item: CRT2} and the assumption that $\alpha \in \Qc_I$.
By Lemma~\ref{lemma: finite field}\eqref{item: finite field 3},
we have
\[ \lcm(1,(\text{order of $X$ in $G_i$})_{i \in I})=\lcm\left(1, \left(\frac{n}{\gcd(n,s_i)}\right)_{i \in I}\right)=\frac{n}{\gcd(n, \gcd(s_i)_{i \in I})}.   \]
Combining the two equations above, we get
\[|\{\alpha, X \alpha, \ldots, X^{n-1}\alpha\}|=\frac{n}{\gcd(n, \gcd(s_i)_{i \in I})}.  \]
Hence it suffices to show that there exists a unique  $k \in \{0, \ldots, \frac{n}{\gcd(n, \gcd(s_i)_{i \in I})}-1\}$
for  which $\psi_I(X^k\alpha)$ is contained in $\Fc$, or equivalently,
\[  \sum_{z \in \Z_n} z f_{X^k\alpha}(z) = 0 \quad (\text{mod }n). \]

By Lemma~\ref{lemma: a_i b_i}\eqref{item: a_i b_i 2}, we have,  for any $k\geq 0$,
\begin{align*}
\phi_{I}((a_i(X^k\alpha))_{i \in I})=\phi_I((k+ a_i(\alpha))_{i \in I}).
\end{align*}
It then follows from the definition of $h_{i,I}$ and $\phi_{i,I}$ that, for any $i \in I$,
\begin{equation}\label{equation: detailed}
\phi_{i,I}(X^k\alpha)= k\,h_{i,I}+\phi_{i,I}(\alpha).
\end{equation}
We then have,  for any $k\geq 0$,
\begin{align}\label{equation: function sum}
\begin{split}
&\sum_{z \in \Z_n} z f_{X^k\alpha}(z)=\sum_{i \in I}s_i \phi_{i,I}(X^k\alpha) \quad \text{(mod $n$)} \qquad \text{(by part \eqref{item: subset sum 1})}\\
=&\sum_{i \in I}  s_i  \left( k\, h_{i,I} + \phi_{i,I}(\alpha) \right)  \quad \text{(mod $n$)} \qquad \text{(by \eqref{equation: detailed})}\\
=& k \sum_{i \in I}  s_i   h_{i,I} + \sum_{i \in I}s_i\phi_{i,I}(\alpha)  \quad \text{(mod $n$)}\\
=&k \gcd(n,\gcd(s_i)_{i\in I}) + \sum_{i \in I}  s_i   \phi_{i,I}(\alpha)   \quad \text{(mod $n$)} \qquad \text{(by  \eqref{equation: definition phi})}.
\end{split}
\end{align}
By the definition of $\gcd$, the sum  $\sum_{i \in I}  s_i   \phi_{i,I}(\alpha)$ is a multiple of 
$\gcd(n,\gcd(s_i)_{i\in I})$ modulo $n$.
Hence there  exists a unique $k\in \{0,1, \ldots, \frac{n}{\gcd(n, \gcd(s_i)_{i \in I})}-1\}$
for which the sum in \eqref{equation: function sum} is equal to $0$.
This completes the proof.
%
%This is equivalent 
%
%
%
%Write $\beta:=X^{k(\alpha)}\alpha$. 
%It follows from \eqref{equation: function sum} and  the choice of $k(\alpha)$ that $\psi_I(\beta)$ is contained in $\Fc$.
% In order to prove that $\beta$ is unique, it suffices to show that 
%
%%Since $\alpha \in \Qc_I$, we have for any $I\subseteq \{1,\ldots,m\}$,
%
%It then follows from \eqref{equation: order} that there exists a unique element in $\{\alpha, X\alpha,  \ldots, X^{n-1}\alpha\}$
% for which its image is contained in $\Fc$.
% This proves part~\eqref{item: a_i b_i 2}. 
%\end{enumerate}
\end{proof}

\begin{definition}\label{definition: bijection}
Let $\widehat{\psi}:\Nc \to \Fc$
be the map defined by
\[ \{\alpha, X\alpha,  \ldots, X^{n-1}\alpha\} \mapsto \psi_{I}(\beta), \]
where $I$ is the subset of $\{1,\ldots,m\}$ such that $\alpha \in \Qc_I$, and $\beta$ is the unique element of  $\{\alpha, X \alpha,  \ldots, X^{n-1}\alpha\}$
for which its image is contained in $\Fc$. 
\end{definition}
%Note that $\psi$ is well defined by Lemma~\ref{lemma: function sum}\eqref{item: subset sum 2}.

\begin{proof}[Proof of Theorem~\ref{theorem: bijection}]
Note that  the maps $\psi_I$ $(I \subseteq \{1,\ldots,m\})$ satisfy \eqref{item: C1} and \eqref{item: C2} by 
Lemma~\ref{lemma: bijection} and Lemma~\ref{lemma: function sum}\eqref{item: subset sum 2}, respectively.
It then follows  that  the map $\widehat{\psi}$ in Definition~\ref{definition: bijection} is a bijection.
\end{proof}

\begin{example}\label{example: 5}
Continuing from Example~\ref{example: 4},
the map $\widehat{\psi}: \Nc \to \Fc$ is  given by (recall that we represent a function 
 $f:\Z_3 \to \{0,1\}$ 
as the set  $\{ z \in \Z_3 \mid f(z)=1 \}$):
\begin{itemize}
\item $\{0\}$ is being mapped to $\psi_{\varnothing}(0)=\{0,1,2\}$;

\item  $\{ 1+X+X^2  \}$ is being mapped to  $\psi_{\{1\}}(1+X+X^2)= \{1,2  \}$;
\item $\{ 1+X,X+X^2, 1+X^2  \}$ is being mapped to  $\psi_{\{2\}}(X+X^2)= \{0 \}$;
\item 
$\{ 1,X, X^2  \}$ is being mapped to  $\psi_{\{1,2\}}(1)= \varnothing$. \qedhere
% \{1+X,X+X^2,1+X^2  \}  & \mapsto \psi_{\{2\}}(X+X^2) =\{(0,1),(1,0),(2,0)  \};\\
% \{1,X,X^2  \}   &\mapsto \psi_{\{1,2\}}(1)=  \{(0,0),(1,0),(2,0)  \}. \qedhere
\end{itemize}
\end{example}

%\begin{proof}[Proof of Theorem~\ref{theorem: bijection}]
%The theorem follows  from Lemma~\ref{lemma: bijection} and Lemma~\ref{lemma: function sum}\eqref{item: subset sum 2}.
%\end{proof}

\section{Some open bijective problems}\label{section: unsolved questions}
We conclude with two bijective problems that refine Theorem~\ref{theorem: combinatorial} and Theorem~\ref{theorem: bijection}.

\begin{enumerate}
\item Construct  a bijection between  $\Nc$ and  $\Fc$ for any  two coprime positive integers $q$ and $n$.
Note that the bijection   in Theorem~\ref{theorem: bijection}
relies on viewing the color for neckaces in $\Nc$
as being drawn from the finite field $\mathbb{F}_q$,
and thus fails to work when $q$ is not a prime power.

\item Let $n$ be an odd positive integer, and let  $k \in \{0,\ldots,n\}$.
Give a bijective proof that these two sets have the same cardinality:
\begin{itemize}
\item The set $\Nc_k$ of necklaces of length $n$ with $k$ black beads and $n-k$ white beads; and
\item The set $\Fc_k$ of functions $f:\Z_n\to \{0,1\}$ such that the sum $\sum_{z \in \Z_n} z f(z)$ is equal to $0$ modulo $n$
and the set $\{z \in \Z_n \mid  f(z)\neq 0  \}$ has cardinality $k$.
\end{itemize}

One can show that  $\Nc_k$ and $\Fc_k$ have the same cardinality
by computing  $|\Nc_k|$  and $|\Fc_k|$ separately.
The cardinality of $\Nc_k$ was computed by \cite[Theorem~1.20]{ACH15} by using the orbit-counting theorem,
and the cardinality of $\Fc_k$ can be computed by using the counting method developed in~\cite{KP93}.
The same bijective problem was asked in \cite{ACH15} for the case that $k$ divides $n$.

We remark that the bijection in 
Theorem~\ref{theorem: bijection} does not map $\Nc_k$ to $\Fc_k$, as can be seen from Example~\ref{example: 5}.

\end{enumerate}

\section*{Acknowledgement}
The author would like to thank Richard Stanley for sharing his knowledge on the status of the problem; Marcelo Aguiar, Henk D.L. Hollmann,  and Lionel Levine for their invaluable advice and encouragement;  
Lila Greco, Viktor Kiss, Jos\'e Bastidas Olaya, Connor Simpson, Karl Thomas B{\aa}{\aa}th Sj\"oblom,
and Lilla T\'othm\'er\'esz for constructive criticism of the paper;
and the anonymous referee for their careful reading and insightful comments.

%\tableofcontents
\bibliographystyle{alpha}
\bibliography{hyperplane}
\end{document}